\documentclass[a4paper,12pt]{amsart}
\usepackage{amssymb,amsfonts,amsmath}
\usepackage{ifthen}
\usepackage{graphicx}
\usepackage{float}
\newtheorem{theorem}{Theorem}[section]
\newtheorem{lemma}[theorem]{Lemma}
\newtheorem{corollary}[theorem]{Corollary}

\newtheorem{remark}[theorem]{Remark}

\numberwithin{equation}{section}

\newcounter{minutes}
\divide\time by 60
\newcounter{hours}
\multiply\time by 60
\addtocounter{minutes}{-\time}

\newcommand{\D}{{\mathbb D}}

\newcommand{\real}{{\operatorname{Re}\,}}

\newcommand{\ds}{\displaystyle}

\begin{document}

\bibliographystyle{amsplain}

\title[On generalized coefficient functionals]%
{On coefficient functionals associated with the Zalcman conjecture}

\def\thefootnote{}
\footnotetext{ \texttt{\tiny File:~\jobname .tex,
          printed: \number\day-\number\month-\number\year,
          \thehours.\ifnum\theminutes<10{0}\fi\theminutes}
} \makeatletter\def\thefootnote{\@arabic\c@footnote}\makeatother

\author[S. Agrawal]{Sarita Agrawal}
\address{S. Agrawal, Discipline of Mathematics,
Indian Institute of Technology Indore,
Simrol, Khandwa Road, Indore 452 020, India}
\email{saritamath44@gmail.com}

\author[S. K. Sahoo]{Swadesh Kumar Sahoo}
\address{S. K. Sahoo, Discipline of Mathematics,
Indian Institute of Technology Indore,
Simrol, Khandwa Road, Indore 452 020, India}
\email{swadesh@iiti.ac.in}

\begin{abstract}
We consider certain subfamilies, of the family of univalent functions in the open unit disk, 
defined by means of sufficient coefficient conditions
for univalency. This article is devoted to studying the problem of the well-known 
conjecture of Zalcman consisting of a generalized coefficient functional, 
the so-called generalized Zalcman conjecture problem, for functions belonging to
those subfamilies.
We estimate the bounds associated with the generalized coefficient functional and show that the 
estimates are sharp.
\\
 

\end{abstract}

\maketitle
\pagestyle{myheadings}
\markboth{S. Agrawal and S. K. Sahoo}{On coefficient functionals associated with 
the Zalcman conjecture}

\section{Introduction}
Denote by $\mathcal{A}$, the class of all analytic functions $f$ in 
$\D:=\{z\in \mathbb{C}:|z|<1\}$ of the form
$$ f(z)=z+\sum_{n=2}^\infty a_n z^n.
$$
Denote by $\mathcal{S}$, the class of {\em univalent functions} in $\mathcal{A}$. 
Then $|a_2^2-a_3|\le 1$ holds for $f\in \mathcal{S}$, see \cite[Theorem~1.5]{Pom75}. 
At the end of 1960's, Zalcman made a conjecture that each $f\in \mathcal{S}$ 
satisfies the inequality 
\begin{equation}\label{Intro-eqn1}
|a_n^2-a_{2n-1}|\le (n-1)^2, \quad n\ge 2
\end{equation} 
with equality for the Koebe function $k(z)=z/(1-z)^2$. 
One of the main aims of the Zalcman conjecture was to prove the Bieberbach conjecture: $|a_n|\le n$, for $n\ge 2$, when $f\in \mathcal{S}$,
using the famous Hayman Regularity Theorem (see \cite[Theorem~5.6, pp. 163]{Dur83}). 
The Bieberbach conjecture was a challenging open problem for function theorists for several decades and was finally settled by de Branges \cite{deB85} in $1984$. 

The problem (\ref{Intro-eqn1}) has been studied for several well-known subclasses of the class $\mathcal{S}$. For example, in \cite{BT86}, Brown and Tsao proved that (\ref{Intro-eqn1}) holds for the class $\mathcal{T}$
of typically real functions and the class $\mathcal{S}^*$ of starlike functions. 
In \cite{Ma88}, Ma proved the Zalcman conjecture for the class $\mathcal{K}$ of close-to-convex functions when $n\ge 4$. 
Readers can refer to, for instance, \cite{ALP14,Kru95,Kru10,LP14} and references therein for more information on this topic.     
A generalized version of Zalcman's inequality, in terms of the so-called 
{\em generalized coefficient functional $\lambda a_n^2-a_{2n-1}$}, $\lambda>0$, has been considered in \cite{ALP14,BT86,EV,LP14}.

In \cite{Ma99}, Ma proposed a generalized version of the Zalcman conjecture as follows:
for $f\in \mathcal{S}$, 
$$|a_n a_m-a_{n+m-1}|\le (n-1)(m-1) \quad (n,m=2,3,\ldots)
$$
and proved that this holds for starlike functions and univalent functions with 
real co-efficients.
In this paper, we establish sharp estimates of the Zalcman conjecture in the form
proposed by Ma in \cite{Ma99} for some subclasses of $\mathcal{S}$. Consequently, we obtain sharp estimates of the results 
proved in \cite{EV} for remaining ranges of $\lambda$. 

\section{Preliminaries and Main results}
We use the concept of {\em convex hull of a set} in this paper, but mainly for 
the set $\mathcal{C}$ of convex functions. 
Denote by $co(\mathcal{C})$, the {\em convex hull of $\mathcal{C}$} and its closure is denoted by $\overline{co(\mathcal{C})}$ 
in the topology of uniform convergence on compact subsets of $\D$. 

A function $f \in \mathcal{A}$ is said to be {\em starlike of order $\beta\, (0\le \beta<1)$} if $\real\{zf'(z)/f(z)\}>\beta$ and denote the class of starlike functions of order $\beta$ by $\mathcal{S}^*(\beta)$.  Similarly, a function $f \in \mathcal{A}$ is said to be {\em convex of order $\beta\, (0\le \beta<1)$} if $\real\{1+zf''(z)/f'(z)\}>\beta$ and denote the class of convex functions of order $\beta$ by $\mathcal{C}(\beta)$. Clearly, functions in the classes $\mathcal{S}^*(\beta)$ and $\mathcal{C}(\beta)$ are univalent in $\D$. Moreover $\mathcal{S}^*(0)=\mathcal{S}^*$ and $\mathcal{C}(0)=C$.

A function $f$ is said to be {\em uniformly starlike} in $\D$ if
$f$ is starlike and has the property that for every circular arc $\gamma$ contained
in $\D$, with center $\zeta \in \D$, the arc $f(\gamma)$ is starlike with respect to
$f(\zeta)$. We denote by $\mathcal{UST}$, the class of all uniformly starlike functions. 
Similarly, we say a convex function $f$ in $\D$ is {\em uniformly convex} if for each 
circular arc $\gamma$ in $\D$ with center $\eta \mbox{ in } \D$, the image arc $f(\gamma)$ is convex. Denote the class of all uniformly convex functions by 
$\mathcal{UCV}$, see \cite{Goo91-1,Goo91-2}. We call a function $f\in \mathcal{A}$
is {\em $\nu$-spiral-like of order $\beta, 0\le \beta <1$}, if there is a real number $\nu\,(-\pi/2< \nu<\pi/2)$ such
that $\real [e^{i\nu}\{zf'(z)/f(z)\}]>\beta \cos \nu$ for $z\in \D$. We denote by $\mathcal{S}_p^\nu(\beta)$, the class of
$\nu$-spiral-like functions of order $\beta$, see \cite{KO02}.
More literature on spiral-like functions can be found in \cite{AS91,Lib67,MA81}. 

Recently, in \cite{EV}, Efraimidis and Vukoti{\'c} have studied the generalized Zalcman
coefficient functional for the subclasses, $\overline{co(\mathcal{C})}$, $\mathcal{R}$ and $H$ of $\mathcal{S}$, where 
the classes $\mathcal{R}$ and $H$ are respectively known as {\em the  Noshiro-Warschawski 
class} and {\em the Hurwitz class}, defined by
$$\mathcal{R}=\{f\in \mathcal{A}:\real{f'(z)}>0\}
$$
and 
$$H=\left \{f\in \mathcal{A}:f(z)=z+\sum_{n=2}^{\infty}a_n z^n \mbox{ and } \sum_{n=2}^{\infty}n|a_n|\le 1\right \}.
$$
A well-known fact is that  
$$H \subset \mathcal{R}\cap \mathcal{S}^* \subseteq \mathcal{S},
$$ 
where the inclusion relation $H \subset \mathcal{R}$ is explained in \cite{EV}.
Now we recall

\medskip
\noindent
{\bf Theorem~A. }\cite[Theorem~3]{EV} {\em Let $0<\lambda\le 2$. If $f\in \overline{co(C)}$, then 
$|\lambda a_n^2-a_{2n-1}|\le 1$ for all $n\ge 2$. For any fixed $n$ and $\lambda<2$, 
equality holds only for the functions of the following form (and for their rotations):
\begin{equation}\label{eqn-thmA}
f(z)=\sum_{k=1}^{2n-2}m_k \frac{z}{1-e^{i\theta_k}z},
\end{equation}
where $0\le m_k\le 1$, $\theta_k=\frac{(2k+1)\pi}{2n-2}$, and 
$$\sum_{k=1}^{n-1}m_{2k}=\sum_{k=1}^{n-1}m_{2k-1}=\frac{1}{2}.
$$}

\medskip
\noindent {\bf Theorem~B. }\cite[Theorem~4]{EV} {\em If $0<\lambda\le 4/3$ and $f\in \mathcal{R}$, then 
for all $n\ge 2$ we have
$$|\lambda a_n^2-a_{2n-1}|\le \frac{2}{2n-1}.
$$
For $\lambda<4/3$ and for any fixed $n\ge 2$, 
equality holds only for the functions of the following form (and for their rotations):
$$F(z)=-z+2\int_0^z \frac{f(t)}{t}dt
$$
where $f(z)$ is given by (\ref{eqn-thmA}).}

\medskip
\noindent
{\bf Theorem~C. }\cite[Theorem~6]{EV}\label{Zac-H}
{\em If $\lambda >0$ and $f\in H$, then for each $n\ge 2$ we have 
$$|\lambda a_n^2-a_{2n-1}|\le \max \left\{\frac{\lambda}{n^2}, \frac{1}{2n-1}\right\}.
$$
Equality holds if and only if
$$f(z)=
\left \{
\begin{array}{ll}
z+\ds\frac{\alpha}{2n-1}z^{2n-1} & \mbox{ for } \lambda \le \ds\frac{n^2}{2n-1},\\
z+\ds\frac{\alpha}{n}z^n & \mbox{ for }\lambda \ge \ds\frac{n^2}{2n-1},\\
\end{array}
\right.
$$
where $\alpha$ is a complex number such that $|\alpha|=1$.} 

\medskip
We intend to extend Theorems A and B in terms of the generalized Zalcman conjecture, 
in the form suggested by Ma in \cite{Ma99}, for the classes $\overline{co(C)}$ and $\mathcal{R}(\beta):=\{f\in \mathcal{A}:\real{f'(z)}>\beta\}$ respectively, where $\beta\in [0,1)$. Note that $\mathcal{R}=\mathcal{R}(0)$.
\begin{theorem}\label{Zac-coc}
If $f\in \overline{co(C)}$, then 
$$|\lambda a_n a_m-a_{n+m-1}|\le \lambda -1,
$$
where $n,m=2,3,\ldots$ and $\lambda \in [2, \infty)$.
Equality holds for the function $l(z)=z/(1-z)$ and its rotations.
\end{theorem}

\begin{theorem}\label{Zac-NW}
If $f\in \mathcal{R}(\beta)$, then 
$$|\lambda a_n a_m-a_{n+m-1}|\le \frac{4\lambda(1-\beta)^2}{nm}-\frac{2(1-\beta)}{n+m-1},
$$
where $n,m=2,3,\ldots$ and $\lambda \in \left[\frac{nm}{(1-\beta)(n+m-1)}, \infty\right)$.
Equality holds for the function $m(z)=-2(1-\beta)\ln{(1-z)}-z(1-2\beta)$ and its rotations.
\end{theorem}

\subsection{The class $\mathcal{H}$}
Define the class 
$$\mathcal{H}=\left \{f\in \mathcal{A}:f(z)=z+\sum_{n=2}^{\infty}a_n z^n \mbox{ and } \sum_{n=2}^{\infty}r(n) |a_n|\le 1, r(n)> 0 \mbox{ for } n\ge 2\right \}.
$$
Here is a partial list of restrictions on $r(n)$ such that $\mathcal{H}$ is a subclass of 
$\mathcal{S}$. 
For example, 
\begin{itemize}
\item If $r(n)=(n-\beta)/(1-\beta)$, then $\mathcal{H}\subset \mathcal{S}^*(\beta)\subset \mathcal{S}$ \cite{Sil97}. In particular, for $\beta=0$ we have $\mathcal{H}=H$, the Hurwitz class.
\item If $r(n)=n(n-\beta)/(1-\beta)$, then $\mathcal{H}\subset \mathcal{C}(\beta)\subset \mathcal{S}$ \cite{Sil97}. 
\item If $r(n)=3n-2$, then $\mathcal{H}\subset \mathcal{UST}\subset \mathcal{S}$ \cite{KC13}.
\item If $r(n)=n(2n-1)$, then $\mathcal{H}\subset \mathcal{UCV}\subset \mathcal{S}$ \cite{KC13}.
\item If $r(n)=n/(1-\beta)$, then $\mathcal{H}\subset \mathcal{R}(\beta)\subset \mathcal{S}$.
\item If $r(n)=1+[(n-1)/(1-\beta)]\sec \nu$, then $\mathcal{H}\subset \mathcal{S}_p^\nu(\beta) \subset \mathcal{S}$ \cite{KO02}.
\end{itemize}
In all these classes $\beta\in [0,1)$.
We now state our main result for the class $\mathcal{H}$.

\begin{theorem}\label{Zac-Hi}
Let $\lambda >0$ and $n=2,3,\ldots$. For $f\in \mathcal{H}$, we have
$$|\lambda a_n^2-a_{2n-1}|\le \max \left\{\frac{\lambda}{r(n)^2}, \frac{1}{r(2n-1)}\right\}.
$$
Equality holds if and only if
$$f(z)=
\left \{
\begin{array}{ll}
z+\ds\frac{\alpha}{r(2n-1)}z^{2n-1} & \mbox{ for } \lambda \le \ds\frac{r(n)^2}{r(2n-1)},\\
z+\ds\frac{\alpha}{r(n)}z^n & \mbox{ for }\lambda \ge \ds\frac{r(n)^2}{r(2n-1)},\\
\end{array}
\right.
$$
where $\alpha$ is a complex number such that $|\alpha|=1$. 
\end{theorem}

We remark that for the choice $r(n)=n$, Theorem~\ref{Zac-Hi} turns into Theorem~C. 
Indeed, our proof is much simpler than the proof of \cite[Theorem~6]{EV}.

\section{Proof of the main results}
This section is devoted to the proof of our main results.  
The following lemmas are useful.

\noindent{\bf Lemma~A. }\cite[Lemma~1]{Liv69}\label{lem-pm}
{\em Let $\mu(\theta)$ be a probability measure on $[0,2\pi]$. 
Then
$$|b_{n-1} b_{m-1}-b_{n+m-2}|\le 2 \quad\quad (n,m = 2,3,\ldots),
$$
where $b_n=2\int_0^{2\pi}{e^{in{\theta}}{\rm d}\mu({\theta})}$.
}

\begin{lemma}\label{lem1}
Let $\lambda\in\mathbb{C}$, $\mu(\theta)$ be a probability measure on $[0,2\pi]$, and for some function $s(n)>0$, write $a_n=s(n) \int_0^{2\pi} {e^{i(n-1){\theta}}{\rm d}\mu({\theta})}=s(n)b_{n-1}/2$ where $b_n$ is same as in Lemma~A. Then
$$|\lambda a_n a_m-a_{n+m-1}|\le \left |\lambda-\frac{2s(n+m-1)}{s(n)s(m)}\right| s(n)s(m)+ s(n+m-1), 
$$
for $n,m = 2,3,\ldots$.
\end{lemma}
\begin{proof}
Putting the values of $a_n,a_m, a_{n+m-1}$ and by using Lemma~A we get
\begin{eqnarray*}
&&|\lambda a_n a_m-a_{n+m-1}|\\
&=&\left |\left(\lambda-\frac{2s(n+m-1)}{s(n)s(m)}\right)s(n)\frac{b_{n-1}}{2}s(m)\frac{b_{m-1}}{2}-\frac{s(n+m-1)}{2}(b_{n-1}b_{m-1}-b_{n+m-2})\right|\\
&\le& \left |\lambda-\frac{2s(n+m-1)}{s(n)s(m)}\right| s(n)s(m)+ s(n+m-1).
\end{eqnarray*}
The proof of our lemma is complete.
\end{proof}

\begin{remark}
Lemma~\ref{lem1} helps us to estimate the generalized 
Zalcman coefficient functional $\lambda a_{n}a_{m}-a_{n+m-1}$
for several classes of functions in $\mathcal{S}$, where
the coefficients $a_n$ are of the form
$s(n) \int_0^{2\pi} {e^{i(n-1){\theta}}{\rm d}\mu({\theta})}$ and
these lead to extremal functions whose series representations are 
of the form $z+\sum_{n=2}^\infty s(n)z^n$,
for instance, see \cite{Ma99} and the present paper. 
\end{remark}

\begin{proof}[\bf Proof of Theorem~\ref{Zac-coc}]
By a well-known result from \cite{BMW71}, there is a unique probability measure $\mu$ on $[0, 2\pi]$, such that
$$f(z)=\int_0^{2\pi}\frac{z}{1-e^{i\theta}z}{\rm d}\mu{(\theta)}
$$
for all $f$ in $\overline{co(C)}$. Comparing the $n$-th coefficients of the series 
expansion of $f$ and of the geometric series expansion of the right hand side, it can easily be seen that
$$a_n=\int_0^{2\pi}{e^{i(n-1){\theta}}{\rm d}\mu({\theta})}, \quad n\ge2.
$$
From Lemma~\ref{lem1}, we can see that $s(n)=1$ and hence we get
$$|\lambda a_n a_m-a_{n+m-1}|\le |\lambda-2|+1=\lambda-1.
$$
The sharpness can easily be verified using the function $l(z)$ stated in the statement.
\end{proof}
\noindent Putting $m=n$ in Theorem~\ref{Zac-coc}, we get
\begin{corollary}\label{cor1}
If $f\in \overline{co(C)}$, then 
$$|\lambda a_n^2-a_{2n-1}|\le \lambda -1,
$$
where $n=2,3,\ldots$ and $\lambda \in [2, \infty)$.
Equality holds for the function $l(z)=z/(1-z)$ and its rotations.
\end{corollary}
\begin{remark}
We can also prove Corollary~\ref{cor1} by using the same technique as in \cite{EV}.
Indeed, from the proof of \cite[Theorem~3]{EV} we have
$$|\lambda a_n^2-a_{2n-1}|\le (\lambda-2)\int_0^{2\pi} \cos^2((n-1)\theta)
{\rm d}\mu(\theta)+1\le \lambda-1,
$$
where the second inequality follows from the fact that $\cos \theta\le 1$ for $0\le \theta\le 2\pi$.
\end{remark}

\begin{proof}[\bf Proof of Theorem~\ref{Zac-NW}]
By the Herglotz representation theorem for functions with positive real part \cite[1.9]{Dur83}, there is a unique probability measure $\mu$ on $[0, 2\pi]$ such that
$$\frac{f'(z)-\beta}{1-\beta} =\int_0^{2\pi}\frac{1+e^{i\theta}z}{1-e^{i\theta}z}{\rm d}\mu{(\theta)}
$$
or, equivalently,
$$ 1+\sum_{n=2}^{\infty}{n a_n z^{n-1}}=1+(1-\beta)\sum_{n=2}^{\infty}{2\int_0^{2\pi}e^{in\theta}{\rm d}\mu{(\theta)}z^n}.
$$
Comparing the coefficients, we obtain
$$a_n=\frac{2(1-\beta)}{n}\int_0^{2\pi}{e^{i(n-1){\theta}}{\rm d}\mu({\theta})}, \quad n\ge2.
$$
From Lemma~\ref{lem1}, we can see that $s(n)=2(1-\beta)/n$ and hence we get
\begin{eqnarray*}
|\lambda a_n a_m-a_{n+m-1}|
&\le &\left|\lambda-\frac{nm}{(1-\beta)(n+m-1)}\right|\frac{4(1-\beta)^2}{nm}+\frac{2(1-\beta)}{n+m-1}\\
&=& \frac{4\lambda (1-\beta)^2}{nm}-\frac{2(1-\beta)}{n+m-1}.
\end{eqnarray*}
The sharpness can easily be verified using the given function $m(z)$ stated in the
hypothesis of the theorem.
\end{proof}

In particular when $m=n$, Theorem~\ref{Zac-NW} leads to
\begin{corollary}\label{cor3.4}
If $f\in \mathcal{R}(\beta)$ and $\lambda \in \left[\frac{n^2}{(2n-1)(1-\beta)}, \infty\right)$,
then 
$$|\lambda a_n^2-a_{2n-1}|\le \frac{4\lambda (1-\beta)^2}{n^2}-\frac{2(1-\beta)}{2n-1},
$$
where $n=2,3,\ldots$. 
Equality holds for the function $m(z)=-2(1-\beta)\ln{(1-z)}-z(1-2\beta)$ and its rotations.
\end{corollary}

\begin{remark}\label{R1}
Alternative proof of Corollary~\ref{cor3.4} can be done by the same technique as in \cite{EV}. 
Indeed, from the proof of \cite[Theorem~4]{EV} we have
\begin{eqnarray*}
|\lambda a_n^2-a_{2n-1}|&\le &\left(\frac{4\lambda (1-\beta)^2}{n^2}-\frac{4(1-\beta)}{2n-1}\right) \int_0^{2\pi} \cos^2((n-1)\theta){\rm d}\mu(\theta)+\frac{2(1-\beta)}{2n-1}\\
&\le & \frac{4\lambda (1-\beta)^2}{n^2}-\frac{2(1-\beta)}{2n-1},
\end{eqnarray*}

where the second inequality follows from the fact that $\cos \theta\le 1$ for $0\le \theta\le 2\pi$.
\end{remark}

\begin{remark}\label{R2}
From Remark~\ref{R1}, it is clear that for $f\in\mathcal{R}(\beta)$ and for $0<\lambda\le 4/3(1-\beta)$,
$$|\lambda a_n^2-a_{2n-1}|\le \frac{2(1-\beta)}{2n-1}.
$$
Equality holds for the function $m(z)$ and its rotations.
\end{remark}

To prove the generalized Zalcman problem for 
$\mathcal{H}$, we need the following lemma which is in a similar form of \cite[Lemma~5]{EV}.
\begin{lemma}\label{l2}
Let $\lambda >0, n \ge 2, q(n), q(2n-1) > 0$, and consider the triangular region
$$P=\{(u, v)\in \mathbb{R}^2: u, v \ge 0, q(n) u+q(2n-1)v\le 1\}$$
in the $uv$-plane. Then
$$\max_{(u,v)\in P}(\lambda u^2+v)=\max\left\{\frac{\lambda }{q(n)^2}, \frac{1}{q(2n-1)}\right\},
$$
and the maximum attain only at $(u,v)=(0,1/q(2n-1))$ and $(u,v)=(1/q(n),0)$.
\end{lemma}

\begin{proof}
The function $F(u, v)=\lambda u^2+v$ is readily seen to have no critical points,
so its maximum on the compact set $P$ is achieved on the boundary $\partial P$.
Clearly, $F(0, v)\le \frac{1}{q(2n-1)}$ while $F(u, 0)\le \frac{\lambda}{q(n)^2}$.

Finally, on the third piece of the boundary of $P$ we have $q(n) u+q(2n-1) v=1$. 
Hence the function $F$ on that piece can be seen as a function of one variable
leading to
$$F(u, v)=g(u)=\lambda u^2+\frac{1-q(n) u}{q(2n-1)}.
$$
Since $g^{\prime\prime}({u})=2\lambda > 0$, the above function cannot achieve its maximum
within the interval $\left[0, \frac{1}{q(n)}\right]$. 
Hence, the maximum value can only be achieved at one of the end points of this interval
 and since 
$$g(0)=\frac{1}{q(2n-1)},\quad g\left(\frac{1}{q(n)}\right)=\frac{\lambda}{q(n)^2},
$$
the assertion follows.
\end{proof}

\begin{remark}
$Lemma~\ref{l2}$ can also be proved using graphical solution method from Linear Programming Problem.
Since the conditions $u,v\ge 0, \mbox{ and }q(n) u+q(2n-1)v\le 1$ give a convex triangular region with the vertices $(0, 0), (0, 1/q(2n-1)) \mbox { and } (1/q(n), 0)$, by 
the graphical solution method, maximum value can only be achieved at one of these vertices. 
\end{remark}

\begin{proof}[\bf Proof of Theorem~\ref{Zac-Hi}]
By the definition of the class $\mathcal{H}$, the ordered pair $(|a_n|, |a_{2n-1}|)$
belongs to $P$, where $P$ is defined in Lemma~\ref{l2}.
Hence by using Lemma \ref{l2}, we have
$$|\lambda a_n^2-a_{2n-1}|\le \lambda |a_n|^2+|a_{2n-1}|\le \max \left\{\frac{\lambda}{r(n)^2}, \frac{1}{r(2n-1)}\right\}.
$$
It can easily be seen that one way implication is true for the equality. 
For the converse part, we must have $r(n)|a_n|+r(2n-1)|a_{2n-1}|=1$. 
Together with the definition of $\mathcal{H}$, it follows that the rotated 
function must be of the form $f_c(z)=z+A_nz^n+A_{2n-1}z^{2n-1}$, 
where, $f_c(z)=\overline{c}f(cz), |c|=1,$ a rotation of a function $f$ in $\mathcal{S}$, $A_n=c^{n-1}a_n$, and similarly $A_{2n-1}=c^{2n-2}a_{2n-1}$. 
Further inspection of the case of equality in Lemma \ref{l2} and the values of 
$\lambda$ readily yields that one of the coefficients $A_n, A_{2n-1}$ must 
be zero and the more precise form of these functions follows immediately.  
This completes the proof of the theorem.
\end{proof}

\subsection*{Concluding remarks}
The generalized Zalcman conjecture in the form proposed by Ma in \cite{Ma99} is still 
open for the classes $\overline{co(C)}$ and $\mathcal{R}(\beta)$ for $0<\lambda<2$ and
$0<\lambda <\frac{nm}{(1-\beta)(n+m-1)}$ respectively. It would also be interesting to
investigate this problem for the class $\mathcal{H}$.

\vskip 1cm
\noindent
{\bf Acknowledgement.} The work of the first author is supported by University
Grants Commission, New Delhi (grant no. F.2-39/2011 (SA-I)). 
The authors would like to
thank Prof. S. Ponnusamy for helpful discussions and suggestions in this topic.

\end{document}